\documentclass[11pt]{amsart}
\usepackage[dvipsnames]{xcolor}
\usepackage{amssymb,amsmath,amsthm,enumerate,mathtools,stmaryrd}
\usepackage[new]{old-arrows}
\usepackage{tikz-cd} 
\usepackage{comment}

\usepackage{mathptmx}
\usepackage[utf8]{inputenc}
\usepackage{bm}			

\usepackage{hyperref}
\hypersetup{%
	colorlinks = true,
	linkcolor = BrickRed,
	citecolor = Green,
	urlcolor   = blue,
  filecolor = red,
}

\usepackage{cleveref}
\usepackage[inline]{enumitem}
\usepackage[margin=0.9in]{geometry}
\usepackage{parskip}

\usepackage[backend=biber,style=alphabetic,doi=false,isbn=false,url=false,eprint=true,maxbibnames=10,minbibnames=3,mincitenames=3,maxcitenames=3,maxalphanames=3,minalphanames=3,backref=true]{biblatex}
\DeclareFieldFormat[misc]{title}{\mkbibquote{#1}}
\DeclareFieldFormat{extraalpha}{#1}
\DeclareLabelalphaTemplate{%
  \labelelement{%
    \field[final]{shorthand}
    \field{label}
    \field[strwidth=3,strside=left,ifnames=1]{labelname}
    \field[strwidth=1,strside=left]{labelname}
  }
}
\DefineBibliographyStrings{english}{%
  backrefpage={},
  backrefpages={}
}

\usepackage{xpatch}
\DeclareFieldFormat{backrefparens}{\addperiod#1}
\xpatchbibmacro{pageref}{parens}{backrefparens}{}{}

\DeclareMathAlphabet{\mathcal}{OMS}{cmsy}{m}{n} 

\newtheorem{thm}{Theorem}[section]
\newtheorem{lem}[thm]{Lemma}
\newtheorem{obs}[thm]{Observation}
\newtheorem{prop}[thm]{Proposition}
\newtheorem{cor}[thm]{Corollary}
\newtheorem{por}[thm]{Porism}

\theoremstyle{definition}
\newtheorem{rem}[thm]{Remark}
\newtheorem{ex}[thm]{Example}

\numberwithin{equation}{section}

\crefname{thm}{theorem}{theorems}
\crefname{rem}{remark}{remarks}
\crefname{prop}{proposition}{propositions}
\crefname{lem}{lemma}{lemmas}
\crefname{por}{porism}{porisms}
\crefname{cor}{corollary}{corollaries}
\crefname{identity}{identity}{identities}
\crefname{equation}{}{}
\crefformat{section}{\S#2#1#3}

\AtBeginEnvironment{rem}{%
  \pushQED{\qed}%
}
\AtEndEnvironment{rem}{\popQED}

\AtBeginEnvironment{ex}{%
  \pushQED{\qed}%
}
\AtEndEnvironment{ex}{\popQED}

\DeclareMathOperator{\htt}{ht} 
\DeclareMathOperator{\rank}{rank} 
\DeclareMathOperator{\Hom}{Hom}
\DeclareMathOperator{\Ext}{Ext}
\DeclareMathOperator{\GL}{GL}
\DeclareMathOperator{\ann}{ann}
\DeclareMathOperator{\diag}{diag}
\DeclareMathOperator{\II}{I}
\DeclareMathOperator{\CC}{C}
\DeclareMathOperator{\inn}{in}
\DeclareMathOperator{\Cl}{Cl}
\DeclareMathOperator{\Frac}{Frac}
\DeclareMathOperator{\coker}{coker}
\DeclareMathOperator{\length}{length}
\DeclareMathOperator{\tors}{tors}
\DeclareMathOperator{\tf}{tf}
\newcommand{\val}{\nu}

\DeclareRobustCommand{\onto}{\relbar\joinrel\twoheadrightarrow}
\newcommand{\fa}{\mathfrak{a}}
\newcommand{\fb}{\mathfrak{b}}
\newcommand{\fc}{\mathfrak{c}}
\newcommand{\fq}{\mathfrak{q}}
\newcommand{\fp}{\mathfrak{p}}
\newcommand{\fm}{\mathfrak{m}}
\newcommand{\fQ}{\mathfrak{Q}}
\newcommand{\OO}{\mathcal{O}}
\newcommand{\VV}{\mathcal{V}}
\newcommand{\vp}{\varpi}
\renewcommand{\aa}{\pmb{a}}
\newcommand{\cc}{\pmb{c}}
\newcommand{\JJ}{\pmb{J}}
\newcommand{\JJA}{\pmb{J}_{A}}

\newcommand{\Grobner}{Gr\"obner}

\newcommand{\IKM}[1][]{\cite[#1]{IKM:Congruence}}
\newcommand{\deff}[1]{{\color{Cerulean}#1}}
\newcommand{\andd}{\quad\text{and}\quad}
\newcommand{\smatrix}[1]{\left[\begin{smallmatrix} #1 \end{smallmatrix}\right]}

\let\subset\subseteq
\let\supset\supseteq
\let\ge\geqslant
\let\le\leqslant

\let\to\longrightarrow

\begin{document}

\title%
[Congruence modules of determinantal rings]%
{Congruence modules and Wiles defects of \\ determinantal rings of maximal minors}

\author{Kashif Khan}
\author{Aryaman Maithani}
\address{Department of Mathematics, University of Utah, 155 South 1400 East, Salt Lake City, UT~84112, USA}
\email{kkhan@math.utah.edu}
\email{maithani@math.utah.edu}

\thanks{A.M. was supported by NSF grants DMS~2101671 and DMS~2349623 and a Simons Dissertation Fellowship.}



\begin{abstract}
	Let $\OO$ be a discrete valuation ring and $A \coloneqq \OO[X_{m \times n}]/\II_{m}(X)$ the determinantal ring of maximal minors. 
	We consider algebra maps $\lambda \colon A \to \OO$, which is tantamount to choosing rank-deficient matrices $\aa \in \OO^{m \times n}$. 
	Following Iyengar--Khare--Manning, we compute the congruence module and the Wiles defect of $A$ at~$\lambda$, expressing them in terms of the $(m - 1)$-sized minors of $\aa$.
\end{abstract}

\maketitle

\section{Introduction} \label{sec:introduction}

	Let $\OO$ be a discrete valuation ring, 
	and $A$ an $\OO$-algebra equipped with a map
	$\lambda \colon A \onto \OO$ of $\OO$-algebras. 
	We set $\fp \coloneqq \ker \lambda$, and let $\fm$ denote the unique maximal ideal of $A$ containing $\fp$. 
	We say that \deff{$A$ is regular at $\lambda$} if 
	$A_{\fp}$ is regular. 
	Following \IKM, we let $\Psi_{\lambda}(A)$ and $\delta_{\lambda}(A)$ denote the congruence module and the Wiles defect of $A_{\fm}$ at $\lambda$, respectively; 
	see \Cref{sec:background} for the definitions. 
	The purpose of this brief manuscript is to compute them for the determinantal rings of maximal minors: 

	\begin{thm} \label{thm:main}
		Let $\OO$ be a discrete valuation ring. 
		Let $2 \le m \le n$ be integers and 
		$A \coloneqq \OO[X_{m \times n}]/\II_{m}(X)$ the determinantal ring of maximal minors. 
		Consider a point $\aa \in \OO^{m \times n}$ with $\II_{m}(\aa) = 0$;
		such a point defines a map $\lambda_{\aa} \colon A \onto \OO$. 
		The algebra $A$ is regular at $\lambda_{\aa}$ if and only if $\rank \aa = m - 1$, in which case, we have
		\begin{equation*} 
			\Psi_{\lambda_{\aa}}(A) 
			\cong 
			\OO/\II_{m - 1}(\aa)
			\andd
			\delta_{\lambda_{\aa}} (A)
			=
			(n - m) w_{\aa},
		\end{equation*}
		where $w_{\aa} = \length_{\OO} (\OO/\II_{m - 1}(\aa))$.
	\end{thm}
	We prove this theorem in \Cref{sec:main-proof} as \Cref{cor:length-congruence}, 
	recovering \IKM[Proposition~4.2] for $m = 2$.
	The quantity $w_{\aa}$ above may also be interpreted as the smallest valuation of an $(m - 1)$-sized minor of $\aa$; see also \Cref{eq:w-alt} for another reinterpretation.
	We expect the above result to aid in calculations in number theory; 
	the case $m = 2$ already finds application in \IKM[\S5] in computing congruence ideals for Steinberg deformation rings. 

	In \Cref{sec:background}, we recall the relevant definitions and results from \IKM\ that are needed for this paper. 
	In \Cref{sec:setup}, we introduce the auxiliary objects and notation needed for the proof of our result, 
	proving facts about these objects in \Cref{sec:over-fields} under the assumption that we are working over a field. 
	Finally, in \Cref{sec:main-proof}, we adapt these to the case where the base ring is a discrete valuation ring, proving our main result. 

\section*{Acknowledgements}
	
	We thank Srikanth B. Iyengar for several interesting discussions surrounding the material and exposition of the paper, and for suggesting the problem to us.

\section{Background} \label{sec:background}

	We recall the category $\CC_{\OO}$ from \IKM[\S2]:
	the objects are pairs $(A, \lambda)$, where $A$ is a noetherian local $\OO$-algebra equipped with a local map of $\OO$-algebras
	$\lambda \colon A \onto \OO$ such that 
	$A$ is regular at $\lambda$, i.e.,
	$A_{\ker \lambda}$ is a regular ring; 
	the morphisms are the morphisms of $\OO$-algebras over $\OO$. 
	For $c \ge 0$, we denote by $\CC_{\OO}(c)$ the full subcategory of $\CC_{\OO}$ consisting of the pairs $(A, \lambda)$ with $\htt(\ker \lambda) = c$. 
	For an $\OO$-module $V$, we let $\tors(V)$ denote the torsion submodule, and
	$V^{\tf} \coloneqq V/\tors(V)$ the torsionfree quotient. 
	Fixing $(A, \lambda)$ in $\CC_{\OO}$ and setting $\fp \coloneqq \ker \lambda$, 
	we recall the related constructions from ibid.: 
	\begin{enumerate}[label=(\roman*)]
		\item 
		the \deff{congruence module} $\Psi_{\lambda}(A)$
		is the cokernel of the natural map
		$\Ext_{A}^{c}(\OO, A) \to \Ext_{A}^{c}(\OO, \OO)^{\tf}$,
		\item 
		the \deff{cotangent module} of $A$ is $\fp/\fp^{2}$, and we set
		$\Phi_{\lambda}(A) \coloneqq \tors(\fp/\fp^{2})$, 
		and
		\item
		the \deff{Wiles defect} is the integer
		$\delta_{\lambda}(A) \coloneqq 
		\length_{\OO}(\Phi_{\lambda}(A)) - 
		\length_{\OO}(\Psi_{\lambda}(A))$.
	\end{enumerate}

	While we refer the reader to \IKM\ for the significance of the congruence module and the Wiles defect in detecting complete intersections, 
	we now recall the relevant results from that paper that we use for our computations.
	As $(A, \lambda)$ is in $\CC_{\OO}$, the congruence module $\Psi_{\lambda}(A)$ is a cyclic torsion $\OO$-module 
	[ibid.\ Lemma~2.9, Theorem~2.11], and thus is completely determined by its length. 
	Moreover, if $A$ is a complete intersection, this length coincides with that of	$\Phi_{\lambda}(A)$ [ibid.\ Theorem~2.32]. 
	In turn, [ibid.\ Theorem~3.5] yields:
	\begin{thm}[Iyengar--Khare--Manning]
		\label{thm:blackbox} 
		Let $\pi \colon (C, \lambda_{C}) \onto (A, \lambda_{A})$ be a surjection in $\CC_{\OO}(c)$ for some $c \ge 0$. 
		Suppose that $C$ is a complete intersection and $A$ is Cohen--Macaulay. Setting
		$\fc \coloneqq \ann_{C}(\ker \pi)$,
		we have 
		\begin{equation*} 
			\pushQED{\qed} 
			\length_{\OO} \Psi_{\lambda_{A}}(A)
			=
			\length_{\OO} \Phi_{\lambda_{C}}(C)
			-
			\length_{\OO} (\OO/\lambda_{C}(\mathfrak{c})). 
			\qedhere
			\popQED 
		\end{equation*}
	\end{thm}	

	We may extend the above definitions to not necessarily local $\OO$-algebras in the following manner: 
	if $A$ is an $\OO$-algebra with a map $\lambda \colon A \onto \OO$, 
	then there is a unique maximal ideal $\fm$ containing $\ker \lambda$, 
	and $\lambda$ extends to a local map 
	$\lambda_{\fm} \colon A_{\fm} \onto \OO$.
	We then set 
	$\Psi_{\lambda}(A) \coloneqq \Psi_{\lambda_{\fm}}(A_{\fm})$
	and 
	similarly for
	$\delta_{\lambda}(A)$ and $\Phi_{\lambda}(A)$;
	this is unambiguous as $\fm$ is uniquely determined by $A$ and $\lambda$.
	Moreover, if $(C, \lambda_{C}) \onto (A, \lambda_{A})$ is a (not necessarily local) surjection over $\OO$, 
	then letting $\fm_{C}$ and $\fm_{A}$ denote the respective maximal ideals determined by these maps, 
	there is an induced local surjection 
	$C_{\fm_{C}} \onto A_{\fm_{A}}$ over $\OO$. 
	This localisation will put us in a position to apply \Cref{thm:blackbox}; 
	it happens to be the case that the rings $C$ and $A$ we consider 
	are complete intersections and Cohen--Macaulay not just locally, but globally as well. 

	\begin{rem} \label{rem:cotangent-description}
		We recall a convenient description of the cotangent module. 
		Suppose $P = \OO[x_{1}, \ldots, x_{n}]$ is a polynomial ring and $A = P/(f_{1}, \ldots, f_{m})$ a finitely generated $\OO$-algebra. 
		Let $\JJ$ be the \emph{Jacobian matrix} with respect to $\pmb{f}$, i.e., $\JJ$ is the $n \times m$ matrix over $P$ given by 
		$[\partial f_{j}/\partial x_{i}]_{i j}$.
		Given an algebra map $A \onto \OO$ with kernel~$\fp$, the $\OO$-module $\fp/\fp^{2}$ has the presentation
		\begin{equation*} 
			\begin{tikzcd}
				\OO^{m} 
				\arrow[r, "\lambda(\JJ)"] &
				\OO^{n}
				\arrow[r] &
				\fp/\fp^{2} 
				\arrow[r] &
				0,
			\end{tikzcd}
		\end{equation*}
		where $\lambda$ is the induced map $P \onto \OO$. 
		This follows, for example, from \Cite[\S16.1 and Proposition 16.12]{Eisenbud:CA}; 
		we have simply put together the facts that $\Omega_{A / \OO}$ is presented by the Jacobian, and that $\Omega_{A / \OO} \otimes_{A} \OO \cong \fp/\fp^{2}$. 
	\end{rem}

\section{Notation} \label{sec:setup}
	
	Let $\OO$ be
	a field or a discrete valuation ring.
	We fix integers $2 \le m \le n$, and let
	$P \coloneqq \OO[X_{m \times n}]$ denote the standard graded polynomial ring in $m n$ variables. 
	Letting $t \coloneqq m - 1$, we set
	$I \coloneqq \II_{t + 1}(X)$ to be the ideal of maximal minors of $X$, and $A \coloneqq P/I$ the corresponding determinantal ring. 
	Given integers $1 \le a < b \le n$, we define $X_{[a, b]}$ to be the submatrix of $X$ obtained by selecting columns $a$ through $b$ (inclusive). 
	Consider the sequence of elements $\pmb{f} = f_{1}, \ldots, f_{n - t} \in I$ given by
	\begin{equation*} 
		f_{i} \coloneqq 
		\det X_{[i, i + t]}.
	\end{equation*}
	Let $J \coloneqq (f_{1}, \ldots, f_{n - t}) P \subset I$ be the ideal generated by $\pmb{f}$, and $C \coloneqq P/J$ the corresponding quotient ring. 
	In view of \cite[Theorem~7.3.1]{BrunsHerzog:Book}, we note that $\htt(I) = n - t$.
	We shall show that $\htt(J) = \htt(I)$, yielding that $C$ is a complete intersection of the same dimension as $A$. 
	Consider the canonical projection
	\begin{equation*} 
		\pi \colon C \onto A,
	\end{equation*}	
	and let 
	$\fc \coloneqq \ann_{C}(\ker \pi)$ be the annihilator of the kernel, i.e., 
	$\fc = (0 :_{C} \ker \pi)$. 
	For $1 \le i \le n - t + 1$, consider the (prime) ideal $\fQ_{i}$ of $P$ defined by
	\begin{equation*} 
		\fQ_{i} \coloneqq \II_{t}(X_{[i, i + t - 1]}).
	\end{equation*}
	Let $Q_{i} \coloneqq \fQ_{i} C$ be its image in $C$, 
	and $\fq_{i} \coloneqq Q_{i} A$ its image in $A$.
	Diagrammatically, we have
	\begin{equation*} 
		\begin{tikzcd}
			P \arrow[r, two heads]       & C \arrow[r, two heads]     & A       \\[-2em]
			\fQ_{i} \arrow[r, two heads] & Q_{i} \arrow[r, two heads] & \fq_{i}.
		\end{tikzcd}
	\end{equation*}	

	We shall show that $\fc = Q_{2} \cdots Q_{n-t}$ as ideals in $C$, aiding in our computation of $\Psi_{\lambda}(A)$ using \Cref{thm:blackbox}. 
	While we require this result over a DVR, we first prove it over fields, making use of \Grobner\ bases, graded canonical modules, and the varieties defined by the ideals. 

\section{Computations over fields} \label{sec:over-fields}
	
	Throughout this section, we continue with the notation of \Cref{sec:setup} under the additional assumption that $\OO$ is a field. 
	We may then talk about the graded canonical module $\omega_{R}$ of an $\mathbb{N}$-graded noetherian ring $R$ with $[R]_{0} = \OO$. 
	Recall that for such a ring $R$, 
	the \deff{$a$-invariant} $a_{R}$ is the negative of the minimal degree of a nonzero homogeneous element of $\omega_{R}$.
	We first prove that $C$ is a reduced complete intersection. 

	\begin{lem} \label{lem:C-is-ci-over-k}
		The sequence $\pmb{f}$ is a regular sequence in $P$. 
		Thus, $C$ is a complete intersection of dimension $m n - n + t$. 
		Moreover, $C$ is reduced.		
	\end{lem}
	\begin{proof} 
		Consider the lexicographic monomial ordering $<$ on $P$ induced by the lexicographic ordering on the variables.
		The initial monomial of $f_{i}$ is the monomial corresponding to the main diagonal of the defining submatrix $X_{[i, i+t]}$, i.e., 
		\begin{equation*} 
			\inn_{<}(f_{i}) = X_{1,i} \, X_{2, i+1} \,\cdots\, X_{m, i + t}.
		\end{equation*}
		In particular, the supports of the initial monomials are pairwise disjoint;
		thus, $\pmb{f}$ is a \Grobner\ basis generating a radical ideal of the claimed height.
	\end{proof}

	\begin{obs}
		As $C$ is a complete intersection obtained by killing $n - t$ elements of degree $m$, we compute its $a$-invariant to be
		$a_{C} = -m n + (n - t)m = -t m$. 
		Thus, $\omega_{C}(t m) \cong C$. 
		\qed
	\end{obs}

	\begin{prop} \label{prop:prod-contained-in-colon}
		As ideals in $P$, we have $\fQ_{2} \cdots \fQ_{n-t} \subset (J :_{P} I)$.
	\end{prop}
	\begin{proof} 
		We wish to prove the containment of ideals
		\begin{equation*} 
			\fQ_{2} \cdots \fQ_{n-t} I \subset J.
		\end{equation*}
		As the ring extension $P \to P \otimes_{\OO} \overline{\OO}^{\text{alg}}$ is pure, we may assume that the field $\OO$ is a algebraically closed. 
		As $J$ is radical by \Cref{lem:C-is-ci-over-k}, it suffices to prove the containment of algebraic sets
		\begin{equation*} 
			\VV(\fQ_{2} \cdots \fQ_{n-t}) \cup \VV(I) 
			\supset 
			\VV(J),
			\qquad
			\text{equivalently,}
			\qquad
			\VV(J) \setminus \VV(\fQ_{2} \cdots \fQ_{n-t})
			\subset \VV(I). 
		\end{equation*}
		Consider a point $\aa \in \OO^{m \times n}$ such that $\aa \in \VV(J)$ and $\aa \notin \VV(\fQ_{j})$ for all $2 \le j \le n - t$. 
		Let $\cc_{1}, \ldots, \cc_{n}$ be the columns of $\aa$. 
		We wish to show that $\aa \in \VV(I)$. 
		This is tantamount to showing that the column span $\langle \cc_{1}, \ldots, \cc_{n} \rangle$ is of dimension at most $t$.
		We now show that every column $\cc_{i}$ is in 
		$\langle \cc_{2}, \ldots, \cc_{t+1} \rangle$, proving the desired statement:
		The hypothesis on $\aa$ tells us that any $t+1$ consecutive columns are linearly dependent, whereas any $t$ consecutive columns among $\cc_{2}, \ldots, \cc_{n-1}$ are linearly independent. 
		This gives us the claim for $i = 1$ directly, and for $i > t+1$, this follows inductively.
	\end{proof}

	Our next goal is to show that the above containment leads to the equality 
	$Q_{2} \cdots Q_{n-t} = \fc$ in $C$. 
	Note that $\fc$ is tautologically annihilated by $\ker \pi$, and thus is an $A$-module. 
	Moreover, as $C$ is reduced, 
	$\fc \cap \ker \pi = 0$, 
	so the ideal $\fc A = \pi(\fc)$ is isomorphic to $\fc$ as a graded $A$-module.

	\begin{lem} \label{lem:omega-is-c}
		We have $\fc A \cong \omega_{A}(-a_{C}) = \omega_{A}(t m)$ as graded $A$-modules.
	\end{lem}
	\begin{proof} 
		By the preceding discussion, it suffices to show that $\fc \cong \omega_{A}(-a_{C})$ as graded $A$-modules. 
		But this is clear, for both the written graded modules are isomorphic to $\Hom_{C}(A, C)$. 
	\end{proof}

	\begin{rem} \label{rem:isomorphic-ideals}
		Suppose $\fa$ and $\fa'$ are graded ideals of $A$ that are isomorphic as (ungraded) $A$-modules.
		Then, as $A$ is a domain, there exists $u \in \Frac(A)$ such that $\fa = u \fa'$. 
		Thus, $\fa \fb \cong \fa' \fb$ for any ideal $\fb \subset A$. 
		Furthermore, if $\fa$ and $\fa'$ are minimally generated in the same degree, then $\fa$ and $\fa'$ are isomorphic as graded $A$-modules.
	\end{rem}

	\begin{lem} \label{lem:omega-is-prod-q}
		There is a graded isomorphism $\omega_{A}(t m) \cong \fq_{2} \cdots \fq_{n-t}$ of $A$-modules.
	\end{lem}
	\begin{proof} 
		By \cite[Corollary~1.6]{BrunsHerzog:a-invariants}, 
		we know that $\omega_{A}(t m) \cong \fq_{1}^{(n - m)} = \fq_{1}^{n - m}$. 
		Thus, it suffices to show that the ideals 
		$\fa \coloneqq \fq_{1}^{n - m}$ and 
		$\fb \coloneqq \fq_{2} \cdots \fq_{n-t}$ 
		are isomorphic as graded $A$-modules. 
		By \cite[Theorem~7.3.5]{BrunsHerzog:Book}, 
		we know that each $\fq_{i}$ represents the same element of the class group $\Cl(A)$. 
		As each $\fq_{i}$ is a height-one prime, 
		we obtain $\fq_{i} \cong \fq_{1}$ as $A$-modules for all $i$. 
		In view of \Cref{rem:isomorphic-ideals}, it follows that $\fa$ and $\fb$ are isomorphic as graded $A$-modules.
	\end{proof}

	\begin{por}
		For any choice of nonnegative integers 
		$k_{1}, \ldots, k_{n-t+1}$ 
		with $\sum k_{i} = n-m$, 
		we have an isomorphism $\omega_{A}(t m) \cong \prod \fq_{i}^{k_{i}}$ of graded $A$-modules. 
		In particular, the product of ideals is unmixed. 
		\qed
	\end{por}

	\begin{cor} \label{cor:c-is-prod-Q-over-field}
		We have $\fc = Q_{2} \cdots Q_{n - t}$ as ideals in $C$. 
	\end{cor}
	\begin{proof}
		By \Cref{prop:prod-contained-in-colon}, we have the containment
		$Q_{2} \cdots Q_{n - t} \subset \fc$. 
		As noted, both ideals are also $A$-modules, and we know that they are isomorphic as graded $A$-modules in view of \Cref{lem:omega-is-c,lem:omega-is-prod-q}.
		Thus, their Hilbert series coincide, and hence they must be equal as we are working over a field. 
	\end{proof}

\section{Proof of the main theorem} \label{sec:main-proof}
	
	Let $\OO$ be a discrete valuation ring with uniformiser $\vp$, 
	residue field $k \coloneqq \OO/\vp\OO$, 
	and fraction field $K \coloneqq \OO[\vp^{-1}]$. 
	We continue with the same notation as in \Cref{sec:setup}. 
	An overline indicates the same objects being considered mod $\vp$, i.e., over the field $k$. 
	We let $\val$ denote the normalised valuation on $\OO$, so that 
	$\val(r) \coloneqq \sup \{k \ge 0 : r \in \vp^{k} \OO\}$. 

	We obtain the following immediate corollaries of the results in \Cref{sec:over-fields}. 

	\begin{cor} \label{cor:C-is-ci-over-O}
		The sequence $\pmb{f}$ is a regular sequence in $P$, and
		$C \coloneqq P/J$ is a complete intersection of dimension 
		$m n - n + t + 1$. 
		Moreover, $C$ is reduced. 
	\end{cor}
	\begin{proof} 
		By \Cref{lem:C-is-ci-over-k}, 
		we know that $\overline{C} = C/\vp C$ is a reduced complete intersection of the correct dimension, 
		showing that the extended sequence $\pmb{f}, \vp$ is $P$-regular. 
	\end{proof}

	\begin{cor} \label{cor:c-is-prod-Q-over-DVR}
		We have $\fc = Q_{2} \cdots Q_{n - t}$ as ideals in $C$.
	\end{cor}
	\begin{proof} 
		\Cref{cor:c-is-prod-Q-over-field} says that the equality holds after base-changing to either $\overline{C}$ or $C[\vp^{-1}]$, 
		but then equality must hold over $C$ as well.
	\end{proof}

	We now construct smooth $\OO$-points for $A$, i.e.,
	construct $\OO$-algebra maps $\lambda \colon A \onto \OO$ such that $A$ is regular at~$\lambda$. 
	An $\OO$-algebra map 
	$\lambda_{\aa} \colon A \to \OO$ is the same as choosing an 
	$m \times n$ matrix $\aa \in \OO^{m \times n}$ 
	with $\II_{t + 1}(\aa) = 0$, i.e., $\rank \aa \le t$. 
	For $\aa \in \OO^{m \times n}$, 
	we define $w_{\aa} \in \mathbb{N} \cup \{\infty\}$ as 
	\begin{align*} 
		w_{\aa} 
		\coloneqq&\
		\length_{\OO} (\OO/\II_{t}(\aa))
		\\
		=&\ 
		\min \{\val(\det(\aa')) : \text{$\aa'$ is a $t \times t$ submatrix of $\aa$}\}.
	\end{align*}

	As $\vp \notin \fp$, localising at $\fp$ inverts $\vp$, and thus 
	$A_{\fp} \cong (K[X]/\II_{t+1}(X))_{\fp}$. 
	By \cite[Proposition~7.3.4]{BrunsHerzog:Book}, $A$ is regular at $\lambda_{\aa}$ if and only if $\II_{t}(\aa) \neq 0$, i.e., $\rank \aa = t$.
	Assume now that $\rank \aa = t$,
	in which case, we also have
	\begin{equation} \label{eq:w-alt}
		w_{\aa} = \length_{\OO} (\tors(\coker \aa)).
	\end{equation}

	Next, we note that $\length_{\OO}(\Psi_{\lambda_{\aa}}(A))$, $\delta_{\lambda_{\aa}}(A)$, and $w_{\aa}$ are unchanged under the action of
	$\GL_{m}(\OO) \times \GL_{n}(\OO)$, i.e., modifying $\aa$ by (invertible) row and column operations does not change these integers. 
	Thus, we may first assume that the $\OO$-point is in Smith normal form, i.e., is of the form
	\begin{equation*} 
		\begin{bmatrix} 
		D_{t \times t} & O_{t \times (n - t)} \\
		O_{1 \times t} & O_{1 \times (n - t)}
		\end{bmatrix},
		\quad
		\text{where}
		\quad
		D = 
		\diag(\vp^{a_{1}}, \ldots, \vp^{a_{t}}) 
		= 
		\begin{bmatrix} 
		\vp^{a_{1}} & & \\
		 & \ddots & \\
		 & & \vp^{a_{t}}
		\end{bmatrix},
	\end{equation*}
	with $a_{1} \le \cdots \le a_{t}$, 
	and the $O$s are zero matrices of the indicated size.  
	However, this choice of $\OO$-point may not be regular for $C$. 
	Thus, for what is to come, 
	we again use columns operations and assume that the point is
	\begin{equation} \label{eq:desc-a}
		\aa = 
		\begin{bmatrix} 
		D & D & \cdots & D & D_{[1,\; n \bmod t]} \\
		O & O & \cdots & O & O
		\end{bmatrix},
	\end{equation}
	i.e., the last row of $\aa$ is zero, 
	and the first $t$ rows consist of cyclically repeating $D$. 
	We set $\Delta \coloneqq \det D$, and note that $\val(\Delta) = w_{\aa}$. 

	\begin{ex}
		When $m = 2$, the matrix $\aa$ takes the form 
		\begin{equation*} 
			\begin{bmatrix}
				\vp^{a} & \vp^{a} & \cdots & \vp^{a} \\ 
				0 & 0 & \cdots & 0
			\end{bmatrix}.
		\end{equation*}

		When $m = 3$, the matrix $\aa$ takes one of the following forms, depending on whether $n$ is even or odd:
		\begin{equation*} 
		\begin{bmatrix}
			\vp^{a} & 0 & \vp^{a} & 0 & \cdots & \vp^{a} & 0 \\ 
			0 & \vp^{b} & 0 & \vp^{b} & \cdots & 0 & \vp^{b} \\ 
			0 & 0 & 0 & 0 & \cdots & 0 & 0
		\end{bmatrix}
		\quad
		\text{or}
		\quad
		\begin{bmatrix}
			\vp^{a} & 0 & \vp^{a} & 0 & \cdots & \vp^{a} \\ 
			0 & \vp^{b} & 0 & \vp^{b} & \cdots & 0 \\ 
			0 & 0 & 0 & 0 & \cdots & 0 
		\end{bmatrix}.
		\qedhere
		\end{equation*}
	\end{ex}

	We now fix the point $\aa$ once and for all, as in \Cref{eq:desc-a}. 
	We set $\lambda_{A} \coloneqq \lambda_{\aa}$, and let
	$\lambda_{C}$ and $\lambda_{P}$ denote the induced maps from 
	$C$ and $P$, respectively. 
	Thus, we have a commutative diagram
	\begin{equation*} 
		\begin{tikzcd}
			P \arrow[r, two heads] \arrow[rd, "\lambda_{P}"', two heads] & C \arrow[r, two heads] \arrow[d, "\lambda_{C}" description, two heads] & A \arrow[ld, "\lambda_{A} = \lambda_{\aa}", two heads] \\[1em]
			& \OO.  &                                       
		\end{tikzcd}
	\end{equation*}

	As noted earlier, 
	$A$ is regular at $\lambda_{A}$ because $\rank \aa = t$. 
	We will show that $C$ is regular at $\lambda_{C}$ as well. 
	As we shall need the Jacobian for this and computing the cotangent module, 
	we first analyse the partial derivatives of $\pmb{f}$.

	\begin{lem} \label{lem:calculating-derivatives}
		The evaluation $\lambda_{P}(\partial f_{k} / \partial X_{i j})$ is nonzero 
		if and only if 
		$X_{i j} \in \{X_{m, k},\, X_{m, k + t}\}$, 
		in which case, the evaluation equals $\pm \Delta$.
	\end{lem}
	Pictorially, the evaluation is nonzero if and only if $X_{i j}$ is one of the two bottom corner variables of $X_{[k, k + t]}$. 
	\begin{proof} 
		Fixing $k$, we let $X'$ denote the square matrix $X_{[k, k + t]}$, so that $f_{k} = \det X'$. 
		It is clear that $\partial f_{k} / \partial X_{i j}$ is zero if $X_{i j}$ does not appear in $X'$. 
		We may now assume that $X_{i j}$ does appear in $X'$. 
		Recall that if $Y$ is a generic square matrix, 
		then $\partial \det Y/  \partial Y_{i j}$ equals the $(i, j)$-th cofactor of $Y$. 
		Thus, $\lambda_{P}(\partial f_{k}/\partial X_{i j})$ may be computed as the corresponding cofactor of $\lambda_{P}(X')$. 
		As the last row of $X'$ is zero, it is clear that the derivative of $f_{k}$ is zero with respect to any variable not in the last row of $X'$. 
		Similarly, as the first and last columns of $X'$ are identical, 
		the cofactor with respect any variable not in these two columns is again zero. 
		This leaves us only with the variables in the last row that appear in either the first or last column, as claimed in the statement of the lemma. 
		For these two variables, the cofactor computation is clear as the corresponding cofactor matrix is equal to $D$ up to a permutation of the rows.
	\end{proof}

	\begin{por} \label{por:derivatives-max-minors}
		If $g$ is any maximal minor of $X$, then
		$\lambda_{P}(\partial g / \partial X_{i j}) \in \{0, \Delta, -\Delta\}$.
		\qed
	\end{por}

	Let $\JJ$ denote the Jacobian matrix of $\pmb{f}$; see \Cref{rem:cotangent-description} for our conventions. 

	\begin{lem} \label{lem:ideal-of-jacobian}
		As an ideal in $\OO$, we have 
		$\lambda_{P}(\II_{n - t}(\JJ)) = \Delta^{n - t} \OO$. 
	\end{lem}
	\begin{proof} 
		By \Cref{lem:calculating-derivatives}, 
		every entry of 
		$\lambda_{P}(\JJ)$ is divisible by $\Delta$, and hence
		we have $\lambda_{P}(\II_{n - t}(\JJ)) \subset \Delta^{n - t} \OO$. 
		To show equality, it suffices to pick $n - t$ rows of $\JJ$ such that the corresponding square submatrix has determinant $\pm \Delta^{n-t}$. 
		Such a choice of $n - t$ rows corresponds to a choice of $n - t$ variables; 
		we choose $y_{j} \coloneqq X_{m, j}$ for $1 \le j \le n - t$, i.e.,
		the bottom-left variable of the submatrix defining $f_{j}$. 
		Letting 
		$\JJ' \coloneqq [\partial f_{j} / \partial y_{i}]_{i j}$ denote the corresponding submatrix, 
		\Cref{lem:calculating-derivatives} yields
		\begin{equation*} 
			\lambda_{P}(\JJ'_{i j})
			=
			\lambda_{P}\left(\partial f_{j} / \partial y_{i}\right)
			=
			\begin{cases}
				\pm \Delta & \text{ if $i = j$ }, \\
				0 & \text{ if $j > i$ }.
			\end{cases}
		\end{equation*}
		Thus, $\lambda_{P}(\JJ')$ is a triangular matrix, 
		and its determinant may be computed as the product of its diagonal elements, 
		giving us $\det(\lambda_{P}(\JJ')) = \pm \Delta^{n - t}$, as desired.
	\end{proof}

	\begin{cor} \label{cor:coker-J-desc}
		As an $\OO$-module, we have 
		$\coker(\lambda(\JJ)) \cong \OO^{m n - (n - t)} \oplus T$, 
		where $T$ is the torsion submodule. 
		Moreover, $\length_{\OO} T = (n - t) w_{\aa}$.
	\end{cor}
	\begin{proof} 
		\Cref{lem:ideal-of-jacobian} tells us that $\lambda_{P}(\JJ)$ has full rank, so that
		the Smith normal form of $\lambda_{P}(\JJ)$ is 
		$\smatrix{M \\ O}$ where $M$ is an $(n - t) \times (n - t)$ diagonal matrix with nonzero determinant. 
		Thus, the rank of the cokernel is $m n - (n - t)$, 
		and the torsion part is measured by $(n-t)$-sized minors, 
		which is again computed in \Cref{lem:ideal-of-jacobian}. 
	\end{proof}

	\begin{cor} \label{cor:cotangent-C-description}
		The $\OO$-algebra $C$ is regular at $\lambda_{C}$, and
		$\length_{\OO} \Phi_{\lambda_{C}}(C) = (n - t) w_{\aa}$. 
	\end{cor}
	\begin{proof} 
		Let $\fp \coloneqq \ker \lambda_{C}$, so that $V \coloneqq \fp/\fp^{2}$ is the cotangent module of $C$, 
		and we have $\Phi_{\lambda_{C}}(C) = \tors(V)$. 
		Note that $C_{\fp}$ is regular precisely when its embedding dimension is equal to its Krull dimension, i.e.,
		precisely when $\rank V = \htt \fp$. 
		In view of \Cref{cor:C-is-ci-over-O}, we know that $\htt \fp = m n - (n - t)$.
		Both assertions now follow as we have $V \cong \coker(\lambda(\JJ))$ 
		by \Cref{rem:cotangent-description}. 
	\end{proof}

	\begin{lem} \label{lem:cotangent-A-description}
		The equality $\length_{\OO} \Phi_{\lambda_{A}}(A) = (n - t) w_{\aa}$ holds.
	\end{lem}
	\begin{proof}[Sketch] 
		The proof follows similarly as that for $C$. 
		As before, $\Phi_{\lambda_{A}}(A)$ is the torsion submodule of 
		$\coker(\lambda_{P}(\JJA))$, where
		$\JJA$ is the Jacobian matrix with respect to the maximal minors of $X$.
		Noting that $\JJA$ contains $\JJ$ as a submatrix, 
		\Cref{por:derivatives-max-minors} and \Cref{lem:ideal-of-jacobian} give us 
		$\II_{n - t}(\lambda_{P}(\JJA)) =
		\II_{n - t}(\lambda_{P}(\JJ))$, 
		so we are now in the same position as 
		\Cref{cor:coker-J-desc,cor:cotangent-C-description}. 
	\end{proof}

	\begin{lem} \label{lem:colength-c}
		We have $\lambda_{C}(\fc) = \Delta^{n - m} \OO$;
		equivalently, $\length_{\OO}(\OO/\lambda_{C}(\fc)) = (n - m) w_{\aa}$. 
	\end{lem}
	\begin{proof} 
		As $\fc = Q_{2} \cdots Q_{n - t}$ by \Cref{cor:c-is-prod-Q-over-DVR}, 
		it suffices to show that $\lambda_{C}(Q_{i}) = \Delta \OO$ for each $i \in [2, n-t]$. 
		Recalling that $Q_{i} = \II_{t}(X_{[i, i+t-1]})$, we obtain
		$\lambda_{C}(Q_{i}) 
		= \II_{t}(\lambda_{C}(X_{[i, i+t-1]}))$. 
		But $\lambda_{C}(X_{[i, i+t-1]}) = \aa_{[i, i+t-1]}$ consists of $t$ consecutive columns of $\aa$, 
		and, is thus, 
		equal to $\smatrix{D \\ O}$ up to a permutation of columns, giving us 
		$\lambda_{C}(Q_{i}) = (\det D) \OO = \Delta \OO$ as desired.
	\end{proof}

	\begin{cor} \label{cor:length-congruence}
		We have $\Psi_{\lambda_{\aa}}(A) \cong \OO/\II_{t}(\aa)$
		and
		$\delta_{\lambda_{\aa}}(A) = (n - m) w_{\aa}$.
	\end{cor}
	\begin{proof} 
		Applying \Cref{thm:blackbox} to the surjection 
		$\pi \colon (C, \lambda_{C}) \onto (A, \lambda_{A})$, we obtain
		\begin{align*} 
			\length_{\OO} \Psi_{\lambda_{A}}(A)
			&=
			\length_{\OO} \Phi_{\lambda_{C}}(C)
			-
			\length_{\OO} (\OO/\lambda_{C}(\mathfrak{c})) \\
			&= 
			(n - t) w_{\aa}
			- (n - m) w_{\aa} = w_{\aa} = \length_{\OO}(\OO/\II_{t}(\aa)),
		\end{align*}
		where the second equality uses \Cref{cor:cotangent-C-description} and \Cref{lem:colength-c}. 
		As $\OO/\II_{t}(\aa)$ and $\Psi_{\lambda_{A}}(A)$ are cyclic $\OO$-modules with the same length, they are isomorphic. 
		The computation of the Wiles defect follows from \Cref{lem:cotangent-A-description}. 
	\end{proof}

\printbibliography

\end{document}